\documentclass[11pt,twoside,]{amsart}
\usepackage{amsmath, amsthm, amscd, amsfonts, amssymb, graphicx, color}
\usepackage[bookmarksnumbered, plainpages]{hyperref}
\newtheorem{thm}{Theorem}[section]
\newtheorem*{thma}{Theorem A}
\newtheorem*{cora}{Corollary A}
\newtheorem*{corb}{Corollary B}
\newtheorem*{cord}{Corollary D}
\newtheorem*{thmb}{Theorem B}
\newtheorem*{thmc}{Theorem C}
\newtheorem*{thmd}{Proposition D}
\newtheorem*{rema}{Remark}
\newtheorem{cor}[thm]{Corollary}
\newtheorem{lem}[thm]{Lemma}
\newtheorem{prop}[thm]{Proposition}
\newtheorem{defn}[thm]{Definition}

\numberwithin{equation}{section}
\def\pn{\par\noindent}

\begin{document}
\begin{tabular}{c r}
\vspace{-0.6cm}

\end{tabular}
\hskip 1.5 cm
\begin{tabular}{l}

\end{tabular}
\hskip 2 cm
\begin{tabular}{c c}
\vspace{-0.2cm}
\vspace{-1.2cm}
\end{tabular}
\vspace{1.3 cm}


\title{Adjoint groups of $p$-nil rings and $p$-group automorphisms}
\author{Yassine Guerboussa$^*$ and Bounabi Daoud}

\thanks{{\scriptsize
\hskip -0.4 true cm 
\newline Keywords: central automorphisms, radical rings, adjoint groups, finite $p$-groups\\
$*$Corresponding author}}
\maketitle

\begin{abstract} 
We introduce a class of rings, namely the class of  left or right $p$-nil rings, for which the adjoint groups behave regularly.  Every $p$-ring is close to being left or right $p$-nil in the sense that it contains a large ideal belonging to this class. Also their adjoint groups occur naturally as groups of automorphisms of $p$-groups.   These facts and some of their applications are investigated in this paper.    
\end{abstract}

\vskip 0.2 true cm


\pagestyle{myheadings}
\markboth{\rightline {\scriptsize  Y. Guerboussa and B. Daoud}}
         {\leftline{\scriptsize Adjoint groups of $p$-nil rings and $p$-group automorphisms }}

\bigskip
\bigskip


\section{\bf Introduction}
\vskip 0.4 true cm
 Ring theory arises in a natural way in studying the automorphism groups of abelian groups.  Indeed, the automorphism group of an abelian group $G$ is the group of units of the ring of endomorphisms of $G$.  This fact was used by Shoda \cite{Sho} in investigating the structure of such automorphism groups.  While the preceding fact is no longer true if $G$ is not abelian, H. Laue (see \cite{Lau}) observed that there is a general analogue which works for some subgroups of $Aut(G)$.\\
If $N$ is a subgroup of $G$, then  we denote by $End_N(G)$ the set of all endomorphisms $u$ of $G$ such that $x^{-1} u(x) \in N$, for all $x \in G$, and $Aut_N(G)$ denotes the set of automorphisms of $G$ that lie in $End_N(G)$.  Clearly $End_N(G)$ is a monoid under the usual composition of maps. \\
If $N$ is normal and abelian, then it can be viewed as a $G$-module via conjugation $n^x=x^{-1}nx$, $x \in G \mbox{ and } n \in N$.  A derivation of $G$ into $N$ is a mapping $\delta:G\rightarrow N$ such that $\delta(xy)=\delta(x)^y \delta(y)$.  The set $Der(G,N)$ of these derivations is a ring under the addition $ \delta_1+\delta_2(x)= \delta_1(x)\delta_2(x)$ and the multiplication  $ \delta_1 \delta_2(x)= \delta_2(\delta_1(x))$, with  $ \delta_1,\delta_2 \in Der(G,N) \mbox { and }x \in G$.   Now, every endomorphism $u \in  End_N(G)$ defines a derivation $\delta_u(x)=x^{-1} u(x)$ of $G$ into $N$.  And conversely, to each derivation $\delta \in Der(G,N)  $ we can associate an endomorphism $u \in End_N(G) $, with $u(x)=x\delta(x)$.
 
 On the other hand, for any (associative) ring $R$, one can define the circle operation $x\circ y=x+y+xy$.  The set of all elements of $R$ forms a monoid with identity element $0 \in R$ under this operation.   This monoid is called the adjoint monoid of the ring $R$, and the adjoint group $ R^\circ$ of $R$ is the group of invertible elements in this monoid.

\begin{prop}(see \cite{Lau}, Lemma 3.1)
Under the above notation, the mapping $u\mapsto \delta_u  $ is an isomorphism between the monoid $End_N(G)$ and the adjoint monoid of the ring $Der(G,N)$.  In particular it induces an isomorphism between the corresponding groups of invertible elements.
\end{prop}

This relation was applied by H. Laue in \cite{Lau} to prove some analogues of Shoda's results.  In \cite{CM}, F. Catino and M. Miccoli  showed that the main results about $IA$-automorphisms of $2$-generated metabelian groups can be derived in a natural way from Laue's relation, with a considerable gain of clarity.  Note also that it was used implicitly by A. Caranti and S. Mattarei (see \cite{Car}), to prove results about the automorphisms of $p$-groups of maximal class.  Excepting these papers, and in spite of its interest and the considerable work devoted to the group automorphisms, it seems that Laue's relation was ignored completely in the existing litterature.\\     
The results of this paper may be seen as the fruit of further investigations of the interplay between Rings and Group Automorphisms, summarized in Proposition 1.1.\\ 
\\First, let us fix some notation. The letter $p$ denotes a prime number.  Having a group $G$, $\gamma_i=\gamma_i(G)$ and $Z_i=Z_i(G)$ denote respectively the terms of the lower and the upper central series of $G$.  We denote by $\Omega_{\{n\}}(G)$ the set of all elements $x \in G$ such that $x^{p^n}=1$, and $\Omega_n(G)$ denotes the subgroup generated by $\Omega_{\{n\}}(G)$. By $d(G)$ we denote the minimal number of generators of $G$, the (Pr\"{u}fer) rank of $G$  is defined to be 
$$rk(G) = sup\{d(H), H \mbox { a finitely generated subgroup of } G\}, $$
 and the exponent of $G$ is denoted by $exp(G)$.\\
As a special notation, $P(G)$ will denote $\gamma_2(G) G^4$ if $p=2$, and $\gamma_2(G) G^p$ if $p>2$, $S(G)$ denotes $Z(G) \cap P(G)$. 
We denote by $r$ and $s$ the integers defined by $exp(G/G')=p^r$ and $exp(Z(G))=p^s$.\\
\\Let $R$ be a ring.  We say that $R$ is left $p$-nil if every element $x$ of order $p$ ($4$ if $p =2$)  in $R^+$ is a left annihilator of $R$, that is $ p x =0$  ($4x=0$ if $p=2$) implies $xy=0$, for all $y\in R$.  We say that $R$ is right $p$-nil if its opposite ring is left $p$-nil.  The ring $R$ is said to be $p$-nil if it is left and right $p$-nil.   

The first result shows that the $p$-power structure of the adjoint group of such a ring $R$ is very close to that of $R^+$.
\begin{thma}
Let $R$ be a $p$-ring.  If $R$ is left or right $p$-nil, then $\Omega_{\{n\}}(R^\circ)=\Omega_n(R^+)$, for every $n \geq 1$.  In particular we have $\Omega_n(R^\circ)=\Omega_{\{n\}}(R^\circ)$.  
\end{thma}
It follows immediately that
\begin{cora}
Let $R$ be a $p$-ring.  If $R$ is $p$-nil, then $\Omega_{1}(R^\circ) \leq Z(R^\circ)$ ($\Omega_{2}(R^\circ) \leq Z(R^\circ)$, for $p=2$), in other word, $R^\circ$ is $p$-central.
\end{cora}
Moreover if we assume that $R$ is a finite $p$-ring and $R^+$ can be generated by $d$ elements, then every subgroup of $R^\circ$ can be generated by $d$ elements. 

\begin{thmb}
Let $R$ be a finite $p$-ring.  If $R$ is left or right $p$-nil, then $rk(R^\circ) = d(R^{+})$.  In particular, $rk(R^\circ) = d(\Omega_{1}(R^\circ))$.
\end{thmb}
It is conjectured in (\cite{Dick}, see Remark (b) in the last paragraph) that $rk(R^\circ) \leq \alpha. rk(R^{+})$, for any nilpotent finite $p$-ring $R$, with $\alpha =2$ if $p=2$, and $\alpha =1$ if $p>2$  (actually this conjecture is formulated for the class of nil rings and the class of radical periodic rings, whose additive groups have a finite  rank, but from that paper one can reduce it to the class of finite nilpotent $p$-rings).  Thus Theorem B confirmes this conjecture in the class of $p$-nil rings.\\
Note that particularly O. Dickenschied proved the above inequality for finite nilpotent $p$-rings, with $\alpha =3$ if $p=2$, and $\alpha =2$ if $p>2$, using powerful $p$-groups (see \cite{Dick}, Lemma 2.4).  The following corollary generalizes this to the class of all finite $p$-rings, though with an alternative (self contained) proof.
\begin{corb}
Let be $R$ a finite $p$-ring and $P$ a $p$-sylow of $R^\circ$.  Then $rk(P) \leq \alpha. d(R^{+})$, with $\alpha =3$ if $p=2$, and $\alpha =2$ if $p>2$.
\end{corb}

Let $N$ be an abelian normal subgroup of a $p$-group $G$.  It is straightforward to see that the set of derivations in $Der(G,N)$ that are trivial on $\Omega_{1}(N)$ or $\Omega_{2}(N)$ if $p=2$  (which in fact can be identified to the subring $Der(G/\Omega_{1}(N),N)$ or $Der(G/\Omega_{2}(N),N)$ for $p=2$) forms a left $p$-nil ring, it follows that the above theorems apply to the group of automorphisms acting trivially on $G/N$ and $\Omega_{1}(N)$ ($\Omega_{2}(N)$ if $p=2$).  However, we would be more interested to the case of the ring $Hom(G,S(G))$ which is right $p$-nil.  Essentially, this fact can be used to prove the following.

\begin{thmc}
Let $G$ be a finite $p$-group of class $c$.  Then the exponent of $Aut_{P(G)}(G)$ does not exceed $p^{t^2c-t}$, where $t=min\{r,s\}$.  Moreover if $G$ is generated by $d$ elements then $exp(P) \leq p^{t^2c-t +d-1}$ if $p>2$ and $exp(P) \leq p^{t^2c-t +2d-1}$ if $p=2$, for any $p$-subgroup $P$ of $Aut(G)$.
\end{thmc}
As show the automorphism groups of elementary abelian $p$-groups, the $p$-exponent of $Aut(G)$ can not in general be  independent from the number of generators of $G$.\\
\\The following result is known in the litterature (see for instance \cite{Baer}).  It is nice that an elegant proof follows from our approach.  

\begin{thmd}
Let $G$ be an abelian $p$-group of rank $d$ and let $d'$ denote the rank of $P(G)$.  Then every $p$-subgroup $P$ of $Aut(G)$ can be generated by $dd'+\frac{d^2}{4}$ elements if $p>2$, and by $dd'+\frac{3d^2-d}{2}$ elements if $p=2$.   In particular for every such a $P$ we have $d(P) \leq \frac{5d^2}{4}$ if $p>2$, and  $d(P) \leq \frac{5d^2-d}{2}$ if $p=2$ .  
\end{thmd}

\begin{cord}
Let $G$ be a finite $p$-group of rank $k$.  Then every $p$-subgroup $P$ of $Aut(G)$ can be generated by $\frac{9k^2}{4}$ elements if $p>2$, and by $\frac{7k^2-k}{2}$ elements if $p=2$.     
\end{cord}

The remaining part of this paper is divided into two sections, Section 2 is devoted to studying the above class of rings and their adjoint groups.  In Section 3 we shows how these rings can be used to investigate the automorphisms of $p$-groups.

\section{\bf {\bf \em{\bf   Adjoint groups of $p$-nil rings.}}}
\vskip 0.4 true cm
Let us recall the definition.
\begin{defn}
Let $R$ be a ring.  We say that $R$ is left $p$-nil if every element $x$ of order $p$ ( $4$ if $p =2$)  in $R^+$ is a left annihilator of $R$, that is $ p x =0$  ($4x=0$ if $p=2$) implies $xy=0$, for all $y\in R$.  We say that $R$ is right $p$-nil if its opposite ring is left $p$-nil.  The ring $R$ is said to be $p$-nil if it is left and right $p$-nil. 
\end{defn}

For instance, for any ring $R$, the subring $S=pR$ ($S=4R$ if $p=2$) is $p$-nil.  Also, it follows easily that the left and the right annihilators of  $\Omega_1(R^+)$ ($\Omega_2(R^+)$ if $p=2$) are respectively right and left $p$-nil.\\
A ring $R$ is said to be nilpotent of class $n$, if  $R^{n+1}=0$ and $n$ is the least non-negative integer satisfying this property, with $R^k$ denotes the additive subgroup generated by all the products of $k$ elements of $R$.    
\begin{thm}
Let $R$ be a ring with an additive group of finite exponent $p^m$.  If $R$ is  left or right $p$-nil, then $R$ is nilpotent of class at most $m$.  In particular the adjoint group $R^\circ$ is nilpotent of class at most $m$.
\end{thm}
\begin{proof}
Assume that $R$ is left $p$-nil.  We claim that $p^{m-n+1}R^n=0$, for all $n \leq m+1$.  This is obvious for $n=1$.  Now if $x \in R^n$, then by induction $p^{m-n+1}x=0$.  It follows that $p^{m-n}x$ has order $1$ or $p$, therefore $ (p^{m-n}x)y=p^{m-n}(xy)=0$, for all $y\in R$.  This shows that $p^{m-n}R^{n+1}=0$.  Now, for $n=m+1$ we have $R^{m+1}=0$, this prove that $R$ is nilpotent of class at most $m$.  The result follows  for $R$ right $p$-nil by a similar argument.  The second assertion follows from (\cite{KP}, Theorem 1.6.4).  
\end{proof}
\begin{rema}
Note that the bound on the nilpotency class can be improved to ${m/2}+1$ for the even prime.  And the above theorem holds for $p=2$, under the assumption that every element $x \in R$ satisfying $2x=0$ is a left or right annihilator of $R$.
\end{rema}

Since it is obvious that a subring of a left (right) $p$-nil ring is left (right) $p$-nil, it is not clear that this would be true for all the factor rings.  The following lemma shows that this holds for some factors. 
\begin{lem}
If $R$ is a left (right, resp) $p$-nil ring, then the factor ring $R/\Omega_n(R^+)$ is left (right, resp) $p$-nil for all $n \geq 1$.
\end{lem}
\begin{proof}
Assume that $R$ is left $p$-nil.\\ 
Assume first that $p>2$, and  let be $ x \in R $ such that $ px \in \Omega_n(R^+)$.  Then $p^nx \in \Omega_1(R^+)$, and by assumption $ (p^nx)y=p^n(xy)=0 \mbox{, for all } y \in R$.  Therefore $ xy \in \Omega_n(R^+)\mbox{, for all } y \in R$.\\
For $p=2$, if $ x \in R$ such that $ 4x \in \Omega_n(R^+)$, then $2^nx \in \Omega_2(R^+)$, therefore $ (2^nx)y=2^n(xy)=0 $, and so $ xy \in \Omega_n(R^+)\mbox{, for all } y \in R$.\\
 The result follows similarly if $R$ is right $p$-nil.  
\end{proof}

\begin{proof}[Proof of Theorem A]
We denote by $ x^{(k)}$ the $k{th}$ power of $x$ in the adjoint group of $R$.
\\Assume first that $n=1$.  For $p>2$, we have $px= 0$ implies $x^i =0 \mbox{ for } i\geq 2$.  Hence $$ x^{(p)}=\sum_{i\geq 1}\binom{p}{i} x^i =px=0,$$
 and so $x\in \Omega_{\{1\}}(R^\circ)$. Conversely, if $ x^{(p)}=0$ then $$ px= -\sum_{i\geq 2}\binom{p}{i} x^i.$$
 Let $p^m$ be the order of $x$ in $R^+$.  If $m\geq 2$, then $ p^{m-1}x$ has order $p$, hence $ p^{m-1}x^2=0$, and similarly we have $p^{m-2}x^i=0\mbox{, for } i\geq 3$.  Now if we multiply the above equation by $p^{m-2}$, we obtain $$ p^{m-1}x= -\sum_{i\geq 2}\binom{p}{i} p^{m-2} x^i=0$$ 
This contradicts the definition of the order of $x$.  Therefore $m\leq 1$, and so $x \in \Omega_{1}(R^+)$.\\
For $p=2$, $2x= 0$ implies $4x= 0$, thus $x^2=0$.  It follows that $ x^{(2)}=2x+x^2=0$, so $x\in \Omega_{\{1\}}(R^\circ)$.  Conversely, if $ x^{(2)}=0$ then $2x=-x^2$.  Assume that $x$ has order $2^n>2$ in $R^{+}$, then $2^{n-2}x^2=0$, thus $2^{n-1}x= -2^{n-2}x^2=0$, a contradiction.  It follows that $2x=0$.   
\\ Now we proceed by induction on $n$.  If $x \in \Omega_n(R^+)$, then $ px \in \Omega_{n-1}(R^+)$.  This implies that $ x+\Omega_{n-1}(R^+) \in \Omega_1(( R/\Omega_{n-1}(R^+))^+)$.  Lemma 2.3  and the first step imply that $x+\Omega_{n-1}(R^+) \in \Omega_{\{1\}}( (R/\Omega_{n-1}(R^+))^\circ)$.  Hence $ x^{(p)} \in \Omega_{n-1}(R^+) $, and by induction $ x^{(p)} \in \Omega_{\{n-1\}}(R^\circ) $.  Thus  $ x \in \Omega_{\{n\}}(R^\circ) $.  It follows that $\Omega_n(R^+) \subset \Omega_{\{n\}}(R^\circ)$.  The inverse inclusion follows similarly.  
\\Finally, the equality  $\Omega_n(R^\circ)=\Omega_{\{n\}}(R^\circ)$ follows from the fact that $(\Omega_n(R^+))^\circ$ is a subgroup of $R^\circ$ and $\Omega_n(R^\circ)$ is generated by $\Omega_{\{n\}}(R^\circ)$.                                      
\end{proof}

Before proving Theorem B, we need the following lemma.
\begin{lem}
If $R$ is a left $p$-nil $p$-ring, and if $U$ denotes the intersection of the right annihilator of $R$ and $\Omega_{1}(R^{+})$. Then $U$ is a non-trivial ideal and the factor ring $R/U$ is left $p$-nil.
\end{lem}
\begin{proof}
As $R$ is nilpotent, let $n$ denote the largest integer such that $R^n \neq 0$.  Then $0 \neq \Omega_{1}(R^n)$ lies in $U$, so $U$ is not trivial.\\
Let be $x, y \in R$ such that $px \in U$ ($4x \in U$, for $p=2$).  Then $z(px)=pzx=0$ ($z(4x)=4zx=0$, for $p=2$)  for all $z \in R$.  As $R$ is left $p$-nil,  it follows that $zxy = 0$,  for all $z \in R$,  hence $xy$ is a right annihilator of R.  Also we have $px \in \Omega_{1}(R^{+})$ ($px \in \Omega_{2}(R^{+})$, for $p=2$), so $pxy=0$, thus $xy \in U$.        
\end{proof}

\begin{proof}[Proof of Theorem B]
Since $U^+ = U^\circ$, we shall denote both of them by $U$.  We claim that $d(H) \leq d(\Omega_{1}(H))$, for any $ H \leq R^\circ$.
Assume that $R$ is left $p$-nil, and assume the result is true for any such ring of order $<|R|$.  Also, we assume that the result holds for any subgroup with order $ <|H|$.
We have $HU/U \cong H/H \cap U$ is a subgroup of $R^\circ/U =(R/U)^\circ$, from our assumption and Lemma 2.4 it follows that $d(H/H \cap U) \leq d(\Omega_{1}(H/H \cap U))$.  Now if $H \cap U \nleq \Phi (G)$, and since $U \leq \Omega_{1}(Z(R^\circ))$, we can find a maximal subgroup $K \leq H$ and a subgroup $P$ of order $p$ in $H \cap U$ such that $H \cong K \times P$.  By the minimality of $H$, we have $d(H) = d(K)+1 \leq d(\Omega_{1}(K))+1 = d(\Omega_{1}(H))$. Otherwise, we have $d(H)=d(H/H \cap U)$. Let $A$ be the subgroup of $H$ such that $A/H \cap U = \Omega_{1}(H/H \cap U)$.  By induction, if $A<H$ then $d(H) \leq d(A/H \cap U) \leq d(A) \leq d(\Omega_{1}(A))$, and since $\Omega_{1}(H)$ is abelian, it follows that $d(H) \leq d(\Omega_{1}(A)) \leq d(\Omega_{1}(H))$. \\
Now we have to assume that $H/H \cap U = \Omega_{1}(H/H \cap U)$ which is abelian.  For $p=2$, we have $H \leq \Omega_{2}(R^\circ)$, so $H$ is abelian and we are done.  For $p>2$, we have $ [H,H] \leq H \cap U$.  Since $H \cap U$ is central, it follows that $H$ is nilpotent of class $\leq 2$, so $H$ is regular (see \cite{H}, II.10 ).  We have $p^{d(H)} \leq |H:H^p| = |\Omega_{1}(H)| = p^{d(\Omega_{1}(H))}$.\\
Finally, since $\Omega_{1}(R^\circ)$ is abelian, it follows that $d(H) \leq d(\Omega_{1}(H)) \leq d(\Omega_{1}(R^\circ)) = d(R^+)$.        
 \end{proof}
 
\begin{proof}[Proof of Corollary B]
First note that the ideal $U=pR$ ($4R$ if $p=2$) is a left $p$-nil $p$-ring, and $(R/U)^\circ \cong R^\circ /U^\circ$ (note that one can take $U$ to be the left or the right annihilator of $\Omega_{1}(R^{+})$, or $\Omega_{2}(R^{+})$ if $p=2$).  Let $H$ be a $p$-subgroup of $R^\circ$.  Then $d(H) \leq d(H/H \cap U^\circ)+d(H \cap U^\circ)$.  Theorem B implies that $d(H \cap U^\circ) \leq d(U^{+}) \leq d(R^{+})$.  On the other hand, $H/H \cap U^\circ \cong HU^\circ/U^\circ$ is a subgroup of $R^\circ /U^\circ \cong (R/U)^\circ$, so $p^{d(H/H \cap U^\circ)} \leq |(R/U)^\circ| \leq |R/U|$.  Now, if $p>2$ then $|R/U|=p^{d(R^{+})}$, and if $p=2$ then $|R/U| = |R/2R||2R/4R| = p^{d(R^{+})}p^{d((2R)^{+})} \leq p^{2d(R^{+})}$, the result follows.
\end{proof}  
 
\section{\bf {\bf \em{\bf   Applications to $p$-group automorphisms.}}}
\vskip 0.4 true cm
Note that we were motivated by the following result in introducing the class of $p$-nil rings. 
\begin{prop}
Let $G$ be a finite $p$-group, and let be $S = Z(G) \cap P (G)$.  Then\\
(a) the ring $Hom(G,S)$ is right $p$-nil;\\
(b) $ \Omega_{n}(Aut_{S}(G))= \Omega_{\{n\}}(Aut_{S}(G))= Aut_{\Omega_{n}(S)}(G) $; \\ 
(c) the exponent of $Aut_{S}(G)$ is $\leq p^{min\{r,s\}}$;\\  
(d) $Aut_{S}(G)$ is nilpotent of class at most $min\{r,s\}$;\\
(e) the rank of $Aut_{S}(G)$ is equal to $d(G)d({S})$.

\end{prop}
\begin{proof}
(a) Let  be $k,h \in Hom(G,S)$ such that $ph=0$ ($4h=0$ if $p=2$).  Hence $Im(h)$ is an abelian group of exponent $p$ ($4$ if $p=2$), so its kernel contains $P(G)$, and since $S \leq P(G)$ we have $Im(k) \subset ker(h) $.  It follows that $h$ is a right annihilator of the ring $Hom(G,S)$. \\ 
 Observe that the additive group $Hom(G,S)=Hom(G/G', S)$ has exponent $ \leq p^{min\{r,s\}}$ and rank $d(G)d(S)$, now (b) and (c) follow from Theorem A, (d) follows from Theorem 2.2, and (e) follows from Theorem B.\\

\end{proof}
Note that Proposition 3.1 provides a new proof of Theorem 4.8  in \cite{JJ}.
\begin{cor}
If $G$ is a finite $p$-group such that $Z(G) \leq \Phi (G)$, then $Aut_{Z(G)}(G)$ is nilpotent of class at most $min\{r,s\}$, where $exp(G/G\prime)=p^r$ and $exp(Z(G))=p^s$. 
\end{cor}
It is interesting that Proposition 3.1 has a strong implication on the structure of $Aut(G)$, where $G$ is an abelian $p$-group.
For instance if $p>2$, we are speaking about the ring $Hom(G,G^p)$ and about the automorphism group $Aut_{G^p}(G)$.
The quotient $Aut(G)/Aut_{G^p}(G)$ can be embedded as a subgroup of $GL(d,p)$, with $d=d(G)$.  Thus $Aut_{G^p}(G)$ has index at most $p^{\binom{d}{2}}$ in a $p$-sylow of $Aut(G)$, that is a $p$-sylow of $Aut(G)$ contains a large normal subgroup having a very regular structure.  
\begin{proof}[Proof of Proposition D]
We have $K=Aut_{P(G)}(G)$ is the adjoint group of the ring $Hom(G,P(G))=Hom(G,S(G))$.  Let $P$ be a $p$-subgroup of $Aut(G)$.  We have $d(P) \leq d(P/P \cap K) + d(P \cap K)$, since $P \cap K$ is a subgroup of $K$ it follows from Proposition 3.1 (e) that $d(P \cap K) \leq dd'$.\\
For $p$ odd, $P/P \cap K \cong PK/K$ is a $p$-subgroup of $GL(d,p)$, and since every $p$-subgroup of $GL(d,p)$ can be generated by $d^2/4$ (see \cite{Pat}), it follows that $d(P/P \cap K) \leq d^2/4$.  Therefore $d(P) \leq dd'+d^2/4$.\\
For $p=2$, $P/P \cap K$ can be embedded as a $2$-subgroup of $Aut(G/G^4)$,  so we have only to show that the $2$-part of $|Aut(A)|$ is at most $2^{\frac{3d^2-d}{2}}$,  for any abelian group $A$ of rank $d$ and exponent $4$.  Indeed, $Aut(A)/Aut_{A^2}(A)$ is a subgroup of $GL(2,d)$, so the order of one of its $2$-sylow is at most $2^{\frac{d^2-d}{2}}$.  On the other hand $Aut_{A^2}(A)$ is isomorphic to the adjoint group of $Hom(A,A^2)$, which has order $2^{d^2}$, the result follows.    
\end{proof}
\begin{proof}[Proof of Corollary D]
The idea of this proof is taken from (\cite{SS}, Lemma 2.1).\\
Let be $P$ a $p$-subgroup of $Aut(G)$, $A$ a maximal abelian $P$-invariant subgroup of $G$, and $C=C_P(A)$.  It follows from the three subgroup lemma applied in $G \rtimes P$, that $[C,G,A]=1$, thus $[C,G] \leq C_G(A)$.  It follows easily from the maximality of $A$ that $C_G(A)=A$, thus $C$ acts trivially on $A$ and $G/A$, so by Laue's relation it can be embedded in the additive group $Der(G/A,A)$, which embeds in a direct sum of $k$ copies of $A$.  Thus $rk(C) \leq k^2$.  Now, as $P/C$ embeds as a $p$-subgroup of $Aut(A)$, the result follows at once from Proposition D.     
\end{proof}
The remainder of the paper is devoted to prove Theorem C.

\begin{lem}
Given a finite $p$-group $G$. Then $Z(Aut_{P(G)}(G))$ has exponent at most $p^{min\{r,s\}}$.  
\end{lem}
\begin{proof}
Let be $u \in Z(Aut_{P(G)}(G))$.  Since $Aut_{P(G)}(G)$ contains $Inn(G)$, we have $u$ commutes with $Inn(G)$, thus $x^{-1} u(x) \in Z(G)$ and so $x^{-1} u(x) \in Z(G)\cap P (G) $, for all $x \in G$.  It follows from Proposition 3.1 (c) that the order of $u$ is at most $p^{min\{r,s\}}$.
\end{proof}
Note that one can replace $Z(Aut_{P(G)}(G))$ in the above lemma by $Z(P)$, where $P$ is $p$-sylow of $Aut(G)$, for (at least) $p>2$.  Indeed, we claim that if $ u \in Z(P)$ then $x^{-1} u(x) \in \Phi(G)$, for all $x \in G$.  We have $u$ is a $p$-automorphism that acts on the $p$-group $\Omega_{1}(Z(G)\cap \Phi (G))$, so it fixes at least a non-trivial element $z$ in this group.  Let $M$ be a maximal subgroup of $G$, and let $r:G \rightarrow \mathbb{Z}_p$ be a homomorphism with kernel $M$.  Consider the endomorphism $h(x) = z^{r(x)}$, for $x \in G$.  Then $1+h :x \mapsto xh(x)$  is an automorphism of $G$ lying in $Aut_{\Phi(G)}(G)$, so it lies in $P$.  It follows that $u$ commutes with $h$, thus $z^{r(x)} = z^{r(u(x))}$, so $x^{-1} u(x) \in M$, for all $x \in G$.  This is true for any maximal subgroup $M$, and the claim follows.\\

The following lemma may be seen as a generalization of Theorem 3 in \cite{Lieb}, our proof is somewhat different.  Below, for a finite $p$-group $G$, $e(G)$ denotes the integer satisfying $p^{e(G)}=exp(G)$.  
\begin{lem}
Let $G$ be a finite $p$-group of class $c$, let be $r_1(G)= \sum_{i=1}^{c} e(\gamma_i/\gamma_{i+1}) $ and $s_1(G)= \sum_{i=1}^{c} e(Z_i/Z_{i-1})$.  Then $Aut_{\Phi(G)}(G)$ is nilpotent of class at most $min\{r_1,s_1\}-1$.  In particular its class does not exceed $tc-1$.
\end{lem}
Recall first that the lower $p$-central series of a group $G$ is defined by $P_1(G) = G$ and by induction $P_{i+1}(G) = P_i(G)^p[P_i(G),G]$, $i \ge 1$.  And note that $P_2(G)=\Phi(G)$.  The least integer $n$ such that $P_{n+1}(G)=1$ is the $p$-lower length of $G$, and any central series of $G$ having factors of exponent $p$, has length at least $n$.   
\begin{proof}
By (\cite{HB}, Theorem VIII.1.7), if an automorphism $u$ of $G$ acts trivially on $G/P_2(G)$, then it acts trivially on each section $P_{i+1}(G)/P_i(G) $.  Thus $Aut_{\Phi(G)}(G)$ is a stability group of the lower $p$-series.  It follows from a well known result of Kaloujnine (see \cite{H}, Satz III.2.9), that $Aut_{\Phi(G)}(G)$ is nilpotent of class at most $n-1$.  \\
Now we have to connect the upper and the lower central series of $G$ to the above series.  Define the $p$-series of an abelian $p$-group $A$ of exponent $p^m$ by 
$$1 < A^{p^{m-1}} < ...<A^p<A$$ 
This series has length $m$, and factors of exponent $p$.  Using this definition one can refine each factor of the lower and the upper central series, by its $p$-series.  We obtain two central series of $G$ having factors of exponent $p$ and their length are respectively equal to $r_1(G)= \sum_{i=1}^{c} e(\gamma_i/\gamma_{i+1}) $ and $s_1(G)= \sum_{i=1}^{c} e(Z_i/Z_{i-1})$.  As $n \leq min\{r_1,s_1\}$, it follows that $Aut_{\Phi(G)}(G)$ is nilpotent of class at most $min\{r_1,s_1\}-1$.\\  
Finally, by a well known result $exp(Z_i/Z_{i-1}) \leq exp(Z(G))$ and $exp(\gamma_i/\gamma_{i+1}) \leq exp(G/\gamma_{2}) $, it follows that $r_1(G) \leq rc$ and $s_1(G) \leq sc$.  Therefore $Aut_{\Phi(G)}(G)$ is nilpotent of class at most $min\{r,s\}c-1 = tc-1$.    
\end{proof}

\begin{proof}[Proof of Theorem C]
By Lemma 3.3, the exponent of the center of $Aut_{P(G)}(G)$ is $\leq p^t$.  It follows from  Lemma 3.4 that the exponent of  $Aut_{P(G)}(G)$ is at most $(p^t)^{tc-1} = p^{t^2c-t}$.\\
Now let $P$ be a $p$-sylow in $Aut(G)$.  As $P/Aut_{P(G)}(G)$ embeds as $p$-subgroup of $Aut(G/P(G))$, it follows from a result of Horosevskii (see \cite{Isa}, Corollary 3.3) that the exponent of $P/Aut_{P(G)}(G)$ is bounded by $p^{d-1}$ if $p>2$ and by $p^{2d-1}$ if $p=2$.  The result follows.
  
\end{proof}

\vskip 0.4 true cm



\bigskip
\bigskip


{\footnotesize \pn{\bf Yassine Guerboussa}\; \\ {Department of
Mathematics}, {University Kasdi Merbah Ouargla,} {Ouargla, Algeria}\\
{\tt Email yassine\_guer@hotmail.fr}\\
{\footnotesize \pn{\bf Bounabi Daoud}\; \\ {Department of
Mathematics}, {University 
Ferhat Abas,} {Setif, Ouargla}\\
{\tt Email: boun\_daoud@yahoo.com}\\

\end{document}